\documentclass[12pt]{amsart}

\newtheorem{Defi}{Definition}[section]
\newtheorem{thm}[Defi]{Theorem}
\newtheorem{lemma}[Defi]{Lemma}
\newtheorem{lem}[Defi]{Lemma}
\newtheorem{definition}[Defi]{Definition}
\newtheorem{cor}[Defi]{Corollary}

\newtheorem{rem}[Defi]{Remark}

\usepackage{amsfonts}
\def\R{{\rm I\!R}}

\def\trait (#1) (#2) (#3){\vrule width #1pt height #2pt depth #3pt}
\def\qed{\hfill \trait (0.1) (6) (0) \trait (6) (0.1) (0) \kern-6pt
\trait (6) (6) (-5.9) \trait (0.1) (6) (0) }

\usepackage{amsthm}

\def \R{\mathbb{R}}

\def \E{\mathbb{E}}

\def \bf{\textbf}

\def \ni {\noindent}

\begin{document}

\setcounter{page}{1}

%\thispagestyle{firstheadings}

%\def\firstrightmark
 %\begin{flushleft}
% \begin{minipage}[t]{9cm}{\it Random Oper. and
%Stoch. Equ.,  \rm Vol. 11, No. 4, pp. 333--348   (2003) \\
% \copyright \   VSP 2003}\end{minipage}\end{flushleft}

\mbox{}
\title [Existence and Uniqueness   of Mild Solutions to
   NSFDEs driven by a  fBm  ]{
\ni {\bf   Existence and Uniqueness of Mild Solutions to Neutral
SFDEs driven by a  Fractional Brownian Motion  with Non-Lipschitz
Coefficients } }\maketitle
 \hrule
 \bigskip

\bigskip

%\ead{hajjisalahe@gmail.com}

\begin{center}
 {S., HAJJI$^1$  and  E., LAKHEL $^2$}\\
 { \ni $^1$Department of Mathematics, Faculty of Sciences Semlalia, Cadi Ayyad University, 2390 Marrakesh,
 Morocco}\\
 {\ni $^2$National School of Applied Sciences, Cadi Ayyad University, 46000 Safi , Morocco}

\end{center}

\medskip

 \footnotetext[1]{Lakhel E.: e.lakhel@uca.ma (Corresponding author) and Hajji S.: hajjisalahe@gmail.com   }

\medskip

%\noindent{\rm Received for {\sl ROSE}  March  2, 2003}

\bigskip

\noindent {\bf Abstract---}{\rm
 The article presents results on existence and uniqueness of mild solutions to a class of non
linear neutral stochastic functional   differential equations
(NSFDEs) driven   by Fractional Brownian motion  in a Hilbert
space with non-Lipschitzian coefficients. The results are obtained
by using the method of Picard approximation and generalize the
results that were reported by \cite{ boufoussi3}.}

\bf{keyword} Mild solution; Semigroup of bounded linear operator;
Fractional powers of closed operators;  Fractional Brownian
motion; Wiener integral; Mild Solutions.
\\
%\MSC   60H15 \sep 60G15
%\sep 60J65

%%%%%%%%%%%%%%%%%%%%%%%%%%%%%%%%%%%%%%%%%%%%%%%%%%%%%%%%%%%%%%%%%%%%%%%%%%
\section{INTRODUCTION}
%%%%%%%%%%%%%%%%%%%%%%%%%%%%%
Stochastic partial differential equations (SPDEs) are considered
by many authors (  see, for example, \cite{daprato}) where the
random disturbances are described by stochastic integrals with
respect to semimartingales, especially by Wiener processes.
However, the Wiener process is not suitable to represent a noise
process if long-range dependence is modelled (see \cite{peter}).
It is then desirable to replace the Wiener process by  fractional
Brownian motions (fbm).  Over the last years some new techniques
have
  been developed in order to
define stochastic integrals with respect to fBm. The study of
solutions of stochastic equations in infinite-dimensional space
 with a (cylindrical) fractional Brownian motion ( for example, stochastic partial differential equations)
   has been relatively limited. Linear and semilinear equations with additive fractional noise
   (the formal derivative of a fbm)
are considered in  \cite{Grec} , \cite{danc} and the same type of
equation is studied
 recently in \cite{carab}.

 Let us now say a few words on stochastic functional  differential equations (SFDEs) driven by a
 fBm.    SFDEs  arise in many areas of applied mathematics.
 For this reason, the study of this type of equations has been receiving increased attention in the last few years.
   In \cite{Ferrante1}, the authors studied the existence and regularity
of the density  by using Skorohod integral based on the Malliavin
calculus. \cite{Neuenkirch} studied the problem  by using rough
path analysis.
     \cite{Ferrante2} studied the existence and convergence when the delay goes to zero by using the
      Riemann-Stieltjes integral.  Using also the Riemann-Stieltjes integral,
       \cite{boufoussi2} proved the existence and
    uniqueness of a mild solution and studied the dependence of the
    solution on the initial condition in finite and infinite dimensional
    space.

However, in some cases, many stochastic dynamical systems depend
not only on present and past states, but also contain the
derivatives with delays (see, e.g., \cite{kolma} and
\cite{kuang}). Neutral stochastic differential equations with
delays are often used to describe such systems. To the best of our
knowledge, there is only a little systematic investigation on the
study  of mild solutions to neutral SPDEs with delays ( see, e.g.,
\cite{ boufoussi3} and references therein).

%*****************************************************************************************
 In  this paper, motivated by the previous references, we are concerned with the  existence and uniqueness of mild
solutions for a class of neutral functional stochastic
differential equations (FSDEs)   described in the form

\begin{eqnarray}\label{jeq1}
d[x(t)+g(t,x(\rho(t)))]&=&[Ax(t)+f(t,x(\rho(t))]dt+\sigma (t)dB_Q^H(t),\;0\leq t \leq T,\nonumber\\
x(t)&=&\varphi (t), \; -r \leq t \leq 0.
\end{eqnarray}

where $A$ is  the infinitesimal generator of an analytic
semigroup, $(T (t))_{t\geq 0}$, of bounded linear operators in a
separable Hilbert space $X$; $B_Q^H$ is a fractional Brownian
motion   on a  Hilbert space $Y$ (see section 2 below);  $f$ , $g$
 and  $\sigma$ are given functions to be
specified later, $\rho:[0,+\infty)\longrightarrow [-r,+\infty)$ is
a suitable delay function,  $\varphi:[-r,0]\times
\Omega\longrightarrow X$ is  the initial value.

 The gaol of this work is to establish an existence and uniqueness result for
  mild solution of equation (\ref{jeq1}). The results are obtained by imposing a condition on
the non linearities, which is weaker than the classical Lipschitz
condition and generalize the results that were reported by \cite{
boufoussi3}.   Our approach is similar to the one in \cite{mahmud}
and     \cite{bao} in the case of Wiener process. The  rest of
this paper is organized as follows. In Section 2 we give a brief
review and  preliminaries needed to establish our results. Section
3  is  devoted to the study of  existence and uniqueness of mild
solution of (\ref{jeq1}) by using a Picard type iteration.
%****************************************************************************
\section{Preliminaries}

In this section, we introduce notations, definitions and
preliminary results which we require to establish the existence
and uniqueness of a solution of equation (\ref{jeq1}).

Let $(\Omega,\mathcal{F}, \mathbb{P})$ be a complete probability
space.   Consider a time interval $[0,T]$ with arbitrary fixed
horizon $T$ and let $\{\beta^H(t) , t \in [0, T ]\}$ the
one-dimensional fractional Brownian motion with Hurst parameter
$H\in(1/2,1)$. This means by definition that $\beta^H$ is a
centered Gaussian process with covariance function:
$$ R_H(s, t) =\frac{1}{2}(t^{2H} + s^{2H}-|t-s|^{2H}).$$
 Moreover $\beta^H$ has the following Wiener
integral representation:
\begin{equation}\label{rep}
\beta^H(t) =\int_0^tK_H(t,s)d\beta(s)
 \end{equation}
where $\beta = \{\beta(t) :\; t\in [0,T]\}$ is a Wiener process,
and $K_H(t; s)$ is the kernel given by
$$K_H(t, s )=c_Hs^{\frac{1}{2}-H}\int_s^t (u-s)^{H-\frac{3}{2}}u^{H-\frac{1}{2}}du$$
for $t>s$, where $c_H=\sqrt{\frac{H(2H-1)}{\beta (2-2H,H-\frac{1}{2})}}$ and $\beta(,)$ denotes the Beta function. We put $K_H(t, s ) =0$ if $t\leq s$.\\
We will denote by $\mathcal{H}$ the reproducing kernel Hilbert
space of the fBm. In fact $\mathcal{H}$ is the closure of set of
indicator functions $\{1_{[0;t]},  t\in[0,T]\}$ with respect to
the scalar product
$$\langle 1_{[0,t]},1_{[0,s]}\rangle _{\mathcal{H}}=R_H(t , s).$$
The mapping $1_{[0,t]}\rightarrow \beta^H(t)$
 can be extended to an isometry between $\mathcal{H}$
and the first  Wiener chaos and we will denote by
$\beta^H(\varphi)$ the image of $\varphi$ by the previous
isometry.

We recall that for $\psi,\varphi \in \mathcal{H}$ their scalar
product in $\mathcal{H}$ is given by
$$\langle \psi,\varphi\rangle _{\mathcal{H}}=H(2H-1)\int_0^T\int_0^T\psi(s)\varphi(t)|t-s|^{2H-2}dsdt.$$
Let us consider the operator $K_H^*$ from $\mathcal{H}$ to
$L^2([0,T])$ defined by
$$(K_H^*\varphi)(s)=\int_s^T\varphi(r)\frac{\partial K}{\partial
r}(r,s)dr.$$ We refer to \cite{nualart} for the proof of the fact
that $K_H^*$ is an isometry between $\mathcal{H}$ and
$L^2([0,T])$. Moreover for any $\varphi \in \mathcal{H}$, we have
$$\beta^H(\varphi)=\int_0^T(K_H^*\varphi)(t)d\beta(t).$$
%**************************************************************************************
%***************************************************************************************
It follows from \cite{nualart} that the elements of $\mathcal{H}$
may be not functions but distributions of negative order. In order
to obtain a space of functions contained in $\mathcal{H}$, we
consider the linear space $|\mathcal{H}|$ generated by the
measurable functions $\psi$ such that
$$\|\psi \|^2_{|\mathcal{H}|}:= \alpha_H  \int_0^T \int_0^T|\psi(s)||\psi(t)| |s-t|^{2H-2}dsdt<\infty,$$
where $\alpha_H = H(2H-1)$. The space $|\mathcal{H}|$ is a Banach
space with the norm  $\|\psi\|_{|\mathcal{H}|}$ and we have the
following inclusions (see \cite{nualart})
\begin{lemma}\label{lem1}
$$\mathbb{L}^2([0,T])\subseteq \mathbb{L}^{1/H}([0,T])\subseteq |\mathcal{H}|\subseteq \mathcal{H},$$
and for any $\varphi\in \mathbb{L}^2([0,T])$, we have
$$\|\psi\|^2_{|\mathcal{H}|}\leq 2HT^{2H-1}\int_0^T
|\psi(s)|^2ds.$$
\end{lemma}
%*********************************************************
Let $X$ and $Y$ be two real, separable Hilbert spaces and let
$\mathcal{L}(Y,X)$ be the space of bounded linear operator from
$Y$ to $X$. For the sake of convenience, we shall use the same
notation to denote the norms in $X,Y$ and $\mathcal{L}(Y,X)$. Let
$Q\in \mathcal{L}(Y,Y)$ be an operator defined by $Qe_n=\lambda_n
e_n$ with finite trace
 $trQ=\sum_{n=1}^{\infty}\lambda_n<\infty$, where $\lambda_n \geq 0 \; (n=1,2...)$ are non-negative
  real numbers and $\{e_n\}\;(n=1,2...)$ is a complete orthonormal basis in $Y$.
 Let $B^H=(B^H(t))$ be  $Y-$ valued fbm on
  $(\Omega,\mathcal{F}, \mathbb{P})$ with covariance $Q$ as
 $$B^H(t)=B^H_Q(t)=\sum_{n=1}^{\infty}\sqrt{\lambda_n}e_n\beta_n^H(t),$$
 where $\beta_n^H$ are real, independent fBm's. This process is  Gaussian, it
 starts from $0$, has zero mean and covariance:
 $$E\langle B^H(t),x\rangle\langle B^H(s),y\rangle=R(s,t)\langle Q(x),y\rangle \;\; \mbox{for all}\; x,y \in Y
 \;\mbox {and}\;  t,s \in [0,T].$$
In order to define Wiener integrals with respect to the $Q$-fBm,
we introduce the space $\mathcal{L}_2^0:=\mathcal{L}_2^0(Y,X)$  of
all $Q$-Hilbert-Schmidt operators $\psi:Y\rightarrow X$. We recall
that $\psi \in \mathcal{L}(Y,X)$ is called a $Q$-Hilbert-Schmidt
operator, if
$$  \|\psi\|_{\mathcal{L}_2^0}^2:=\sum_{n=1}^{\infty}\|\sqrt{\lambda_n}\psi e_n\|^2 <\infty,
$$
and that the space $\mathcal{L}_2^0$ equipped with the inner
product
$\langle \varphi,\psi \rangle_{\mathcal{L}_2^0}=\sum_{n=1}^{\infty}\langle \varphi e_n,\psi e_n\rangle$ is a
separable Hilbert space.\\

Now, let $\phi(s);\,s\in [0,T]$ be a function with values in
$\mathcal{L}_2^0(Y,X)$.
 The Wiener integral of $\phi$ with respect
to $B^H$ is defined by

\begin{equation}\label{int}
\int_0^t\phi(s)dB^H(s)=\sum_{n=1}^{\infty}\int_0^t
\sqrt{\lambda_n}\phi(s)e_nd\beta^H_n(s)=\sum_{n=1}^{\infty}\int_0^t
\sqrt{\lambda_n}(K_H^*(\phi e_n)(s)d\beta_n(s),
\end{equation}
where $\beta_n$ is the standard Brownian motion used to  present $\beta_n^H$ as in $(\ref{rep})$.\\
Now, we end this subsection by stating the following result which
is fundamental to prove our result. It can be proved by  similar
arguments as those used to prove   Lemma 2 in \cite{carab}.
\begin{lemma}\label{lem2}
If $\psi:[0,T]\rightarrow \mathcal{L}_2^0(Y,X)$ satisfies
$\int_0^T \|\psi(s)\|^2_{\mathcal{L}_2^0}ds<\infty$.  Then the
above sum in $(\ref{int})$ is well defined as a $X$-valued random
variable and
 we have$$ \mathbb{E}\|\int_0^t\psi(s)dB^H(s)\|^2\leq 2Ht^{2H-1}\int_0^t \|\psi(s)\|_{\mathcal{L}_2^0}^2ds.
 $$
\end{lemma}

 Let $A:D(A)\rightarrow X$  be the infinitesimal generator of an
analytic semigroup, $(T (t))_{t\geq 0}$, of bounded linear
operators on $X$. For the theory of strongly continuous semigroup,
we refer to \cite{pazy}.  We will point out here some notations
and properties that will be used in this work. Hence, for
convenience, we suppose that $\|T(t)\|\leq M$ for $t\geq0$,  and
$0 \in \rho(A)$,  where $\rho (A)$ is the resolvent set of $A$,
then it is possible to define the fractional power $(-A)^{\alpha}$
for $0 < \alpha \leq 1$, as a closed linear operator on its domain
$D(-A)^{\alpha}$. Furthermore, the subspace $D(-A)^{\alpha}$ is
dense in $X$, and the expression
$$\|h\|_{\alpha} =\|(-A)^{\alpha}h\|,
$$
 defines a norm in
$D(-A)^{\alpha}$. If $H_{\alpha}$ represents the space
$D(-A)^{\alpha}$ endowed with the norm $\|.\|_{\alpha}$, then the
following properties are well known (cf. \cite{pazy}, p. 74).
\begin{lem}\label{jlem1} Suppose that the preceding conditions are satisfied.\\
(1) Let $0<\alpha \leq  1$. Then $H_{\alpha}$ is a Banach space.\\
(2) If $0 <\beta \leq \alpha $ then the injection $H_{\alpha}
\hookrightarrow
H_{\beta}$ is continuous.\\
 (3) For every $0<\alpha \leq 1$ there exists $C_{\alpha} > 0 $ such that
 $$\| (-A)^{\alpha}T (t)\|\leq\frac{C_{\alpha}}{t^{\alpha}}
 , \;\;0 < t \leq T.$$
 \end{lem}
 Finally, we remark that for the proof of our theorem we need the
following   Bihari's inequality (cf. \cite{biha}).

\begin{lem}\label{elem4}
Let $\rho: \mathbb{R}^+\rightarrow \mathbb{R}^+$ be a continuous
and non-decreasing function and let $g,h,\lambda $ be non-negative
functions on $\mathbb{R}^+$ such that
  $$g(t)\leq h(t)+\int_0^t \lambda(s)\rho (g(s))ds,\;\;t\geq 0,$$
  then $$g(t)\leq G^{-1}\left(G(h^*(t))+\int_0^t \lambda (s)ds\right),$$
where $G(x):=\int_{x_0}^x \frac{1}{\rho (y)}dy$ is well defined
for some $x_0>0$, $G^{-1}$ is the inverse function of $G$ and $
h^*(t):=\sup_{s\leq t}h(s)$. In particular, we have the
Gronwall-Bellman lemma:
\newline If   $$g(t)\leq h(t)+\int_0^t \lambda(s)g(s)ds,$$
  then $$g(t)\leq h^*(t)\exp\left(\int_0^t \lambda (s)ds\right).$$
\end{lem}

\section{The Main Result}

%******************************************************************
In this section we study the existence and uniqueness of mild
 solution of equation (\ref{jeq1}).  Henceforth we will assume that $A$ is the
infinitesimal generator of an analytic semigroup, $(T (t))_{t\geq
0}$, of bounded linear operators on $X$. Further, to avoid
unnecessary notations, we suppose that $0 \in \rho (A)$ and that,
see Lemma \ref{jlem1},

$$ \| T (t)\|\leq M\; \;\;\mbox{and}\;\;  \|(-A)^{1-\beta}T(t)\|\leq \frac{C_{1-\beta}}{t^{1-\beta}} $$ for
some constants $M,\; M_{1-\beta}$ and every $t \in [0,T].$\\
Similar to the deterministic situation we give the following
definition of mild solutions for equation (\ref{jeq1}).

\begin{definition}
A $X$-valued  process $\{x(t),\;t\in[-r,T]\}$, is called  a mild
solution of equation (\ref{jeq1}) if
\begin{itemize}
\item[$i)$] $x(.)\in \mathcal{C}([-r,T],\mathbb{L}^2(\Omega,X))$,
\item[$ii)$] $x(t)=\varphi(t), \, -r \leq t \leq 0$.
\item[$iii)$]For arbitrary $t \in [0,T]$, we have
\begin{eqnarray*}
x(t)&=&T(t)(\varphi(0)+g(0,\varphi(\rho(0))))-g(t,x(\rho(t)))\\
&-& \int_0^t AT(t-s)g(s,x(\rho(s)))ds +\int_0^t T(t-s)f(s,x(\rho (s))ds\\
&+&\int_0^tT(t-s)\sigma(s)dB^H(s)\;\;\; \mathbb{P}-a.s.
\end{eqnarray*}
\end{itemize}
\end{definition}
In order to show the existence and the uniqueness of mild solution
 to equation (\ref{jeq1}),    the following weaker
conditions (instead of the global Lipschitz condition and linear
growth) are listed.

%****************************************************************************
\begin{itemize}
\item [$(\mathcal{H}.1)$]    $f:[0,T]\times X \rightarrow X  $ and $\sigma:[0,T]\longrightarrow \mathcal{L}^0_2(Y,X)$
 satisfying
 the following conditions:
there exists  a function $K:
[0,+\infty)\times[0,+\infty)\rightarrow [0,+\infty)$ such that
\begin{itemize}
\item [$(1a)$]  $\forall t$  $K(t,.)$ is continuous
non-decreasing, concave, and for each fixed $x\in\R_+$,
$\int_0^TK(s,x)ds<\infty$
 \item [$(1b)$] For any fixed $t\in
[0,T]$ and $x\in X$
$$\|f(t,x)\|^2 \leq
K(t,\|x\|^2) \;and \;
\int_0^T\|\sigma(t)\|^2_{\mathcal{L}^0_2}dt<\infty.
$$
\item [$(1c)$] For any constant $\alpha >0,  u_0\geq 0$,  the
integral equation
\begin{equation}\label{jeq0}
u(t)=u_0+\alpha \int_0^t K(s, u(s))ds
\end{equation}
 has a global solution on
$[0,T].$
\end{itemize}
\item [$(\mathcal{H}.2)$] There exists a function $G:[0,+\infty)\times [0,+\infty)\rightarrow [0,+\infty)$ :
\begin{itemize}
\item [$(2a)$] $\forall t\in [0,T]$,  $G(t,.)$ is continuous
non-decreasing and concave with $G(t,0)=0$, and for each fixed
$x\in \R_+$, $\int_0^TG(s,x)ds<+\infty$
 \item [$(2b)$]For any  $t\in [0,T]$ and $x, y\in X$
$$  \|f(t,x)-f(t,y)\|^2  \leq G(t, \|x-y\|^2),$$
\item [$(2c)$]For any constant $D >0$; if a non negative function
$z(t), t\in [0,T]$ satisfies $z(0) =0$ and $z(t)\leq D
\int_0^tG(s,z(s))ds$, then $z(t)=0$ for all $t \in [0,T].$
\end{itemize}
\item [$(\mathcal{H}.3)$]There exist constants $\frac{1}{2} <
\beta < 1,\;  l,\, M_g$ such
 that the function $g$ is $H_{\beta}$-valued,
  $(-A)^{\beta}g:[0,T]\times  X \rightarrow X$ is continuous
 and satisfies
\begin{itemize}
\item [$(3a)$]For all  $t\in [0,T]$ and $x\in X $,
$$\|(-A)^{\beta}g(t,x)\|^2\leq l(\|x\|^2+1).$$ \item
[$(3b)$]For all  $t\in [0,T]$ and $x,y \in X$
$$ \|(-A)^{\beta}g(t,x)-(-A)^{\beta}g(t,y)\| \leq M_g \|x-y\|.$$
\item [$(3c)$] The constants $M_g,\, l$  and $\beta$ satisfy the
following inequalities $$3 \|(-A)^{-\beta}\|^2 M_g^2<1, \;\;5
\|(-A)^{-\beta}\|^2 l<1 .$$
\end{itemize}
\item [$(\mathcal{H}.4)$] $\rho:[0,\infty]\rightarrow \mathbb{R}$ is a continuous
  function satisfying the condition that
$$
 -r\leq\rho(t)\leq t, \; \forall t\geq0.
$$
\end{itemize}
Moreover, we assume that $\varphi \in \mathcal{C}([-r,0],\mathbb{L}^2(\Omega,X))$.\\

%**************************************************************************
The main result of this paper is given in the next theorem.
\begin{thm}\label{jthm1}
Suppose that $(\mathcal{H}.1)$-$(\mathcal{H}.4)$  hold. Then, for
all $T>0$,  the equation (\ref{jeq1}) has a unique mild solution
on  $[-r,T]$.
\end{thm}
%*****************************************
For the proof, we will need the following lemmas.

\begin{lem}\label{nlem4}
Let $\tilde{f} \in L^2([0,T],X), \tilde{\sigma}\in
L^2([0,T],\mathcal{L}^0_2) $, and consider the equation
\begin{eqnarray}\label{neq2}
d[x(t)+g(t,x(\rho(t)))]&=&[Ax(t)+\tilde{f}(t)]dt+\tilde{\sigma} (t)dB^H(t),\;\; 0\leq t\leq T,\nonumber\\
x(&t)=&\varphi(t),\;\; -r\leq t\leq 0.
\end{eqnarray}
Under condition $(\mathcal{H}.3)$ and $(\mathcal{H}.4)$ Equation
(\ref{neq2})  has  a unique mild solution on $[-r,T]$.
\end{lem}

\begin{proof}

Fix $T > 0$ and let $B_T :=
\mathcal{C}([-r,T],\mathbb{L}^2(\Omega, X))$ be the Banach space
of all continuous functions from $[-r, T]$ into
$\mathbb{L}^2(\Omega, X))$, equipped with the supremum norm
$$\|x\|_{B_T}^2=\sup_{-r\leq t \leq T}\mathbb{E}\|x(t,\omega)\|^2.$$

Let us consider the set
 $$S_T=\{x\in B_T : x(s)=\varphi(s),\; \mbox {for} \;\;s \in [-r,0] \}.$$
 $S_T$ is a closed subset of $B_T$ provided with the norm  $\|.\|_{B_T}.$\\
Let $\psi$ be the function defined on $S_T$ by
$\psi(x)(t)=\varphi(t)$ for $t\in [-r,0]$
 and for $t\in [0,T]$
 \begin{eqnarray*}
 &\psi(x)&(t)=T(t)(\varphi(0)+g(0,\varphi(\rho(0))))-g(t,x(\rho(t) ))-\int_0^t AT(t-s)g(s,x(\rho(s)))ds\\
&+&\int_0^t T(t-s)\widetilde{f}(s)ds+\int_0^tT(t-s)\widetilde{\sigma}(s)dB^H(s))\\
&=&  \sum_{i=1}^5I_i(t).
 \end{eqnarray*}

%We will first prove that the function $\psi$ is well defined .\\

We are going to show that each function $t\rightarrow I_i(t)$ is continuous  on $[0,T]$ in the $\mathbb{L}^2(\Omega,X)$-sense.\\
The continuity of $I_1$ follows directly from the continuity of $t\rightarrow T(t)h$.\\
By $(\mathcal{H}.3)$, the function $(-A)^{\beta}g$ is continuous
and since the operator $(-A)^{-\beta}$
 is bounded then $t\rightarrow g(t,x(\rho(t))$ is continuous on $[0,T].$\\
For the third term $I_3(t)=\int_0^t AT(t-s)g(s,x(\rho(s)))ds$, we
have
\begin{eqnarray*}
|I_3(t+h)-I_3(t)|&\leq & \left|\int_0^t (T(h)-I)(-A)^{1-\beta}T(t-s)(-A)^{\beta}g(s,x(\rho(s)))ds\right|\\
&& +\left|\int_t^{t+h} (-A)^{1-\beta}T(t+h-s)(-A)^{\beta}g(s,x(\rho(s)))ds\right|\\
&\leq & I_{31}(h)+I_{32}(h).
\end{eqnarray*}
By the strong continuity of $T(t)$, we have for each $s \in
[0,T]$, $$\lim_{h\rightarrow
0}(T(h)-I)(-A)^{1-\beta}T(t-s)(-A)^{\beta}g(s,x(\rho(s)))=0$$ and
since
 $$
 \|(T(h)-I)(-A)^{1-\beta}T(t-s)(-A)^{\beta}g(s,x(\rho(s)))\|\leq (M+1)\frac{C_{1-\beta}}{(t-s)^{1-\beta}}\|
 (-A)^{\beta}g(s,x(\rho(s)))\|,
 $$
 we conclude by the
 Lebesgue dominated theorem that
 $$
 \lim_{h\rightarrow
 0}I_{31}(h)=0.$$
On the  other hand,
$$|I_{32}(h)|^2 \leq C.|h|^{\beta}\int_0^T (l \|x(\rho(s))\|^2+l)ds;$$
then  $$\lim_{h\rightarrow 0}I_3(t+h)-I_3(t) =0.
$$
 Standard computations can be used to show the continuity of $I_4$. \\
For the term $I_5(h)$, we have
\begin{eqnarray*}
I_5(h)&\leq & \|\int_0^t (T(h)-I)T(t-s)\widetilde{\sigma}(s)dB^H(s)\|\\
&+&\|\int_t^{t+h} T(t+h-s)\widetilde{\sigma}(s)dB^H(s)\|\\
&\leq & I_{51}(h)+I_{52}(h).
\end{eqnarray*}

By   Lemma \ref{lem2}, we get that
\begin{eqnarray*}
E|I_{51}(h)|^2&\leq &2Ht^{2H-1}\int_0^t \|(T(h)-I)T(t-s)\widetilde{\sigma}(s)\|_{\mathcal{L}_2^0}^2ds\\
&\leq & 2HT^{2H-1}M^2\int_0^T
\|(T(h)-I)\widetilde{\sigma}(s)\|_{\mathcal{L}_2^0}^2ds.
\end{eqnarray*}
 Since $\displaystyle\lim_{h\rightarrow 0} \|(T(h)-I)\widetilde{\sigma}(s)\|_{\mathcal{L}_2^0}^2=0$  and
 $$\|(T(h)-I)\widetilde{\sigma}(s)\|_{\mathcal{L}_2^0}^2\leq ( M+1)^2 \|\widetilde{\sigma}(s)\|_{\mathcal{L}_2^0}^2\in \mathbb{L}^1([0,T],ds),$$
 we conclude, by the dominated convergence theorem that,
 $$ \lim_{h\rightarrow 0}\mathbb{E}|I_{51}(h)|^2=0  $$
 Again by Lemma \ref{lem2}, we get that
$$\mathbb{E}|I_{52}(h)|^2\leq 2Hh^{2H-1}M^2\int_t^{t+h} \|\widetilde{\sigma}(s)\|_{\mathcal{L}_2^0}^2ds\rightarrow 0.$$
The above arguments show that $\displaystyle\lim_{h\rightarrow
0}\mathbb{E}\|\psi(x)(t+h)-\psi(x)(t)\|^2=0$.  Hence, we conclude
that  the function  $t \rightarrow \psi(x)(t)$ is continuous on
$[0,T]$ in the $\mathbb{L}^2$-sense.

Next, to see that $\psi (S_T)\subset S_T$, let $x\in S_T$  and $t
\in [0,T]$. We have
\begin{eqnarray*}
\|\psi(x)(t)\|^2&\leq&5\|T(t)(\varphi(0)+g(0,\varphi(\rho(0))\|^2+5\|g(t,x(\rho(t))\|^2\\
&+&5\|\int_0^t AT(t-s)g(s,x(\rho(s)))ds\|^2 +5\|\int_0^t
T(t-s)\widetilde{f}(s)ds\|^2\\
&+&5\|\int_0^tT(t-s)\widetilde{\sigma}(s)dB^H(s))\|^2 \\
&=& 5\sum_{1\leq i \leq 5}J_i(t).
\end{eqnarray*}

 Standard computation yield
$$
\sup_{0\leq t\leq T}\E J_1(t)\leq
M^2\E\|\varphi(0)+g(0,\varphi(\rho(0)))\|^2.
$$

 By using condition $(3a)$ and H\"older's inequality,
we have

$$
\E J_2(t)\leq \|(-A)^{-\beta}\|^2 \left[l
\E\|x(\rho(t))\|^2+l\right],
$$
and hence,
$$
\sup_{0\leq t\leq T}\E J_2(t)\leq \|(-A)^{-\beta}\|^2
 \left[l \sup_{0\leq t\leq T}\E\|x(\rho(t)))\|^2+l\right],
 $$
 Using again condition $(3a)$ and H\"older's inequality,
we have
\begin{eqnarray*}
\E  J_3(t)&\leq&\int_0^t
\|(-A)^{1-\beta}T(t-s)\|^2 ds \int_0^t \E\|(-A)^\beta g(s,x(s))\|^2 ds\\
&\leq& \int_0^t \frac{C^2_{1-\beta}}{(t-s)^{2(1-\beta)}}\int_0^t
(l\E\|x(\rho(s))\|^2+l) ds.\\
 &\leq&
C^2_{1-\beta}\left(\frac{T^{2\beta -1}}{2\beta-1}\right) \int_0^T
(l\E\|x(\rho(s))\|^2+l) ds.
\end{eqnarray*}

Standard computation yield
$$
\sup_{0\leq t\leq T}\E J_4(t) \leq M^2 T
\int_0^T\E\|\widetilde{f}(s)\|^2ds.
$$

By using  Lemma \ref{lem2}, we obtain
$$
\sup_{0\leq t\leq T}\E J_5(t) \leq 2HT^{2H-1}\int_0^T\|
{\widetilde{\sigma}}(s)\|^2_{\mathcal{L}^0_2}ds.
$$

 Since
$\psi(x)(t)=\varphi(t)$ on $[-r,0]$,
 the inequalities together imply that
$$\sup_{-r\leq t \leq T}\mathbb{E}\|\psi(x)(t)\|^2<\infty.$$
 Hence,  we conclude that $\psi$ is well defined.

 Now, we are going to show that $\psi$ is a contraction mapping in
$S_{T_1}$ with some $T_1\leq T$ to be specified later.  Let $x,y
\in S_T$ and $t\in [0,T]$,  we have
\begin{eqnarray*}
\|\psi(x)(t)-\psi(y)(t)\|^2&\leq&2\|g(t,x(\rho(t))-g(t,y(\rho(t))\|^2\\
&+&2\|\int_0^tAT(t-s)(g(s,x(\rho(s)))-g(s,y(\rho(s)))ds\|^2\\
&\leq& 2\|(-A)^{-\beta}\|^2
\|(-A)^{\beta}g(t,x(\rho(t))-(-A)^{\beta}g(t,y(\rho(t))\|^2\\
&+&2\|\int_0^t
(-A)^{1-\beta}T(t-s)(-A)^{\beta}(g(s,x(\rho(s)))-g(s,y(\rho(s))))ds\|^2.
\end{eqnarray*}
By condition $(3b)$, Lemma \ref{jlem1} and H\"older's inequality,
we have
\begin{eqnarray*}
\|\psi(x)(t)-\psi(y)(t)\|^2 &\leq & 2 \|(-A)^{-\beta}\|^2M_g^2
\|x(\rho(t))-y(\rho(t))\|^2\\
&&+ 2M_g^2 C^2_{1-\beta}\left(\frac{t^{2\beta
-1}}{2\beta-1}\right) \int_0^t\|x(\rho(s))-y(\rho(s))\|^2ds.
\end{eqnarray*}
Hence
$$
\sup_{s\in[-r,t]}\mathbb{E}\|\psi(x)(s)-\psi(y)(s)\|^2\leq
\gamma(t)\sup_{s\in[-r,t]} \mathbb{E} \|x(s)-y(s)\|^2.$$ where
$$ \gamma(t)=2M_g^2\left\{\|(-A)^{-\beta}\|^2+C^2_{1-\beta}
\left(\frac{t^{2\beta-1}}{2\beta-1}\right)t\right\}.$$
 By condition $(3c)$, we have
$\gamma(0)=2 \|(-A)^{-\beta}\|^2 M_g^2 <1$. Then there exists
$0<T_1\leq T $ such that $0<\gamma(T_1)<1$ and $\psi$ is a
contraction mapping on $S_{T_1}$ and therefore has a unique fixed
point, which is a mild solution of equation (\ref{neq2}) on
$[0,T_1]$. This procedure can be repeated in order to extend the
solution to the entire interval $[-r,T]$ in finitely  many steps.
\end{proof}
%****************************************************************************

%*****************************************************************************

We now construct a successive approximation sequence using  a
Picard type iteration with the help of Lemma \ref{nlem4}.  Let
$x^0$ be a solution of equation (\ref{neq2})  with $\tilde{f}=0,\,
\tilde{ \sigma }=0$ .  For $n\geq 0$,  let $x^{n+1} $ be the
solution of equation (\ref{neq2}) on $[-r,T]$ with
$\tilde{f}(t)=f(t,x^n(\rho(t)))$, and
$\tilde{\sigma}(t)=\sigma(t)$

 i.e.
\begin{eqnarray}\label{eqn}
x^{n+1}(t)&=&\varphi(t)\qquad\qquad \qquad\mbox{if }\; t\in [-r,0]\nonumber\\
x^{n+1}(t)&=&T(t)(\varphi(0)+g(0,\varphi(0)))-g(t,x^{n+1}(\rho(t)))-\int_0^t AT(t-s)g(s,x^{n+1}(\rho(s)))ds \nonumber\\
&+&\int_0^t T(t-s)f(s,x^n(\rho(s)))ds+\int_0^t
T(t-s)\sigma(s)dB^H(s), \mbox{if} \;t\in [0,T]
\end{eqnarray}

Now, we prove the existence of solution to problem (\ref{jeq1}).
We start the proof by checking the following lemmas.
\begin{lem}\label{jlem7}
Under conditions $(\mathcal{H}.1)-(\mathcal{H}.4)$, the sequence
$\{x^n, n\geq 0\}$ is well defined  and there exist positive
constants $M_1,\, M_2, \,D_0$ such that
 for all $m,n \in \mathbb{N}$ and $t\in [0,T]$
 \begin{enumerate}
\item
 \begin{equation} \label{neq4}
 \sup_{-r \leq s\leq t}\mathbb{E}\|x^{m+1}(s)-x^{n+1}(s)\|^2\leq
  M_1 \int_0^t G(s,\sup_{-r \leq \theta\leq s}\mathbb{E}\|x^m(\theta)-x^n(\theta)\|^2)ds
 \end{equation}
 \item
 \begin{equation}\label{neq5}
\sup_{-r \leq s\leq t}\mathbb{E}\|x^{n+1}(s)\|^2\leq
D_0+M_2\int_0^t K(s,\sup_{-r\leq \theta \leq
s}\mathbb{E}\|x^n(\theta)\|^2)ds.
 \end{equation}
 \end{enumerate}
\end{lem}

%*****************************************************************************

\begin{proof}
$\bullet$ $1$: For $m,n \in \mathbb{N} $ and $t \in [0,T]$ we have
$$\|x^{m+1}(t)-x^{n+1}(t)\|^2\leq  3\left(I_1(t)+I_2(t)+I_3(t)\right)$$
where
$$I_1(t):=\|g(t,x^{m+1}(\rho(t)))-g(t,x^{n+1}(\rho(t)))\|^2,$$
$$I_2(t):=\|\int_0^t AT(t-s)(g(s,x^{m+1}(\rho(s)))-g(s,x^{n+1}(\rho(s))))ds\|^2,$$
\begin{eqnarray*}
I_3(t)&:=&\|\int_0^tT(t-s)(f(s,x^m(\rho(s)))-f(s,x^n(\rho(s))))ds\|^2.
\end{eqnarray*}
By using  condition $(3b)$ for the terms $I_1$ and $I_2$, we
obtain
\begin{eqnarray*}
I_1(t) &\leq & \|(-A)^{-\beta}\|^2 M_g^2\|x^{m+1}(\rho(t))-x^{n+1}(\rho(t))\|^2,\\
 I_2(t) & \leq & \frac{C^2_{1-\beta}}{2\beta-1}T^{2\beta-1}\int_0^t M_g^2\|x^{m+1}(\rho(s))-x^{n+1}(\rho(s))\|^2ds.\\
\end{eqnarray*}
By using  condition $(2b)$ for the term $I_3$,
 we obtain
 \begin{eqnarray*}
\sup_{0 \leq s\leq t}\mathbb{E}I_3(s) &\leq & C \int_0^t
G(s,\sup_{-r \leq \theta \leq
s}\mathbb{E}\|x^m(\theta)-x^n(\theta)\|^2)ds.
\end{eqnarray*}
Using the fact that $3\|(-A)^{-\beta}\|^2 M_g^2 < 1$ and the above
inequalities, we obtain that:
\begin{eqnarray*}
\sup_{-r \leq s\leq t}\mathbb{E}\|x^{m+1}(s)-x^{n+1}(s)\|^2&\leq &C\int_0^t \sup_{-r\leq \theta \leq s}\mathbb{E}\|x^{m+1}(\theta )-x^{n+1}(\theta)\|^2ds\\
&&+C \int_0^t G(s,\sup_{-r\leq \theta \leq
s}\mathbb{E}\|x^m(\theta)-x^n(\theta)\|^2)ds.
\end{eqnarray*}
By Lemma \ref{elem4}, we obtain
$$\sup_{-r \leq s\leq t}\mathbb{E}\|x^{m+1}(s)-x^{n+1}(s)\|^2\leq C \int_0^t G(s,\sup_{-r\leq \theta \leq s}\mathbb{E}\|x^m(\theta)-x^n(\theta)\|^2)ds.$$
 $\bullet$ $2$: By the
same method as in the proof of assertion ($1$), we obtain that
\begin{eqnarray*}
\sup_{-r \leq s\leq t}\mathbb{E}\|x^{m+1}(s)\|^2&\leq &
C +  C\int_0^t \sup_{-r\leq \theta \leq s}\mathbb{E}\|x^{m+1}(\theta)\|^2ds\\
&&+ C \int_0^t K(s,\sup_{-r\leq \theta \leq
s}\mathbb{E}\|x^m(\theta)\|^2)ds.
\end{eqnarray*}
 By Lemma \ref{elem4}, we obtain
$$\sup_{ -r \leq s\leq t}\mathbb{E}\|x^{m+1}(s)\|^2\leq C + C\int_0^t
K(s,\sup_{-r\leq \theta\leq s}\mathbb{E}\|x^m(\theta)\|^2)ds.$$
\end{proof}
%*****************************************************************************************
\begin{lem}\label{jlem6}
Under conditions $(\mathcal{H}.1)-(\mathcal{H}.4)$, there exists
an $u(t)$ satisfying  $$u(t)= u_0+D \int_0^t K(u(s))ds$$ for some
$u_0\geq 0,\; D>0$ and the sequence $\{x^n, n\geq 0\}$ satisfies,
for all $n \in \mathbb{N}$ and $t\in [0,T]$
\begin{equation}\label{jeq6}
\sup_{-r \leq s\leq t}\mathbb{E}\|x^n(s)\|^2\leq  u(t)
\end{equation}
\end{lem}
%
%\bf{Proof}.
\begin{proof}
 Let $u:[0,T]\rightarrow \mathbb{R}$ be a  global solution of the integral equation  (\ref{jeq0}) with an initial
   condition $ u_0=D_0\vee \sup_{-r \leq t\leq T}\mathbb{E}\displaystyle\|x^0(t)\|^2$
  and with $\alpha =M_2$, where $D_0,\, M_2$ are the same constants as in Lemma \ref{jlem7}.  We prove  inequality (\ref{jeq6})  by mathematical induction.\newline
 For $n=0$, the inequality (\ref{jeq6}) holds by the definition of $u_0$.\newline
 Let us assume that $\sup_{-r \leq s\leq t}\mathbb{E}\displaystyle\|x^n(t)\|^2\leq u(t)$. Then, by (\ref{neq5}),
  we obtain
 \begin{eqnarray*}
\sup_{-r \leq s\leq t}\mathbb{E}\|x^{n+1}(s)\|^2&\leq &D_0 +
M_2\int_0^t K(s,\sup_{-r \leq \theta \leq s}\mathbb{E}
\|x^n(\theta)\|^2)ds\\
  &\leq &u_0 + M_2\int_0^t K(s,u(s))ds =u(t).
\end{eqnarray*}
This completes the proof.
%$\square$
\end{proof}

\begin{proof}[Proof of Theorem \ref{jthm1}]\quad\newline
$\bullet$ {Existence}: For $t \in [0,T]$, by Lemma \ref{jlem7} we
note that

\begin{equation}
 \sup_{-r \leq s\leq t}\mathbb{E}\|x^{m+1}(s)-x^{n+1}(s)\|^2\leq
  M_1 \int_0^t G(s,\sup_{-r \leq \theta\leq
  s}\mathbb{E}\|x^m(\theta)-x^n(\theta)\|^2)ds.
 \end{equation}

By Lemma \ref{jlem6} and the Fatou Lemma

$$\limsup_{m,n\rightarrow \infty} (\sup_{ -r \leq s\leq
t}\mathbb{E}\|x^m(s)-x^n(s)\|^2)\leq  M_1 \int_0^t
G(s,\limsup_{m,n\rightarrow \infty}\sup_{-r \leq \theta\leq
  s}\mathbb{E}\|x^m(\theta)-x^n(\theta)\|^2)ds.
$$
  By condition (2c),

$$\lim_{m,n\rightarrow +\infty }\sup_{ -r \leq s\leq
T} \mathbb{E}\|x^m(s)-x^n(s)\|^2=0.
$$
This implies that
$(x^n,\,\;n\geq1)$ is a Cauchy sequence in $B_T$. Therefore, the
completeness of $B_T$ guarantees   the existence of  a process
$x\in B_T$ such that
$$\lim_{n\rightarrow +\infty } \sup_{-r \leq s\leq
T}\mathbb{E}\|x^n(s)-x(s)\|^2=0,
$$

Letting $n\rightarrow +\infty$ in (\ref{eqn}); it is seen that $x$
is a mild solution to equation (\ref{jeq1}) on $[-r,T]$.

 $\bullet$ {Uniqueness}: Let $x$ and $y$ be two mild solutions of equation
(\ref{jeq1}) on $[-r,T]$, then
$$\sup_{-r \leq s\leq t}\mathbb{E}\| x(s)- y(s)\|^2 \leq M_1\int_0^t G(s,\sup_{-r \leq \theta \leq
s}\mathbb{E}\|x(\theta)-y(\theta)\|^2)ds$$ By condition (2c), we
get $\sup_{-r \leq s\leq T} \mathbb{E}\displaystyle\|
x(s)-y(s)\|^2=0$. Consequently, $x=y$ which implies the
uniqueness. The proof of theorem is complete.
\end{proof}

%*************************************************
%*****************************************************

\begin{rem}

  \item [- ]  If $\forall t\in[0,T]$, we have $ G(t,u)=Lu,$ $u\geq0$, $L>0$, condition
$(\mathcal{H}_2)$ implies global Lipschitz condition. We see that
the Lipschitz condition is a special case of the proposed
conditions.
\end{rem}

\begin{cor}
Suppose that $(H.3)-(H.4)$ are satisfied. Further we suppose that
for each fixed $t\in [0,T]$ and $x,y \in X$, the following
conditions are satisfied,
\begin{itemize}
  \item [a.1 ] $\|f(t,x)-f(t,y)\|^2\leq
  \alpha(t)\lambda(\|x-y\|^2)$.
  \item [a.2 ] $\|f(t,0)\|,\,\|\sigma(t)\|_{\mathcal{L}^0_2}\in
  L^2([0,T];\R^+)$ for all $t\in [0,T]$,\\ where $\alpha(t)\geq0$ is  such that $\int_0^T
\alpha(s)ds<+\infty$, and $\lambda:\mathbb{R}_+\rightarrow
\mathbb{R}_+$ is a continuous concave non-decreasing function such
that $\lambda(0)=0$, $\lambda(u)>0$ for $u>0$ and
$\int_{0^+}\frac{1}{\lambda(x)}=+\infty.$ Then equation
(\ref{jeq1}) has a unique solution.
\end{itemize}
 \end{cor}

\begin{rem} A concrete examples of the function $\lambda(.)$.  Let $L>0$ and $\delta\in (0,1)$ be sufficient  small.
 Define
 $$\lambda_1(u)=L u,\, u\geq0 \;
 $$

 $$
 \lambda_2(u)=   \left\{\begin{array}{ll}& ulog(u^{-1}),\;\; 0\leq
u\leq
\delta,\\
&  \delta log(\delta^{-1})+\lambda_{2}^{'}(\delta_-)(u-\delta),\;
\; u>\delta,
\end{array}\right.
$$
where $\lambda_{2}^{'}$ denotes the derivative of function
$\lambda_2$. They are all concave nondecreasing functions
satisfying $\int_{0^+}\frac{1}{\lambda_{i}(x)}=+\infty (i=1,2).$
\end{rem}
%*************************************************************************************


\begin{thebibliography}{30}
\bibitem  {bao}
\textsc {J. Bao     and Z. Hou    .} Existence of mild solution to
stochastic neutral partial functional differential equations with
non-Lipschitz coefficients. Computers and Mth. with Appl., 59
(2011), 207-214.

\bibitem {biha}
 \textsc{Bihari, I.,  1956.} A generalization of a lemma of Belmman
and its application to uniqueness problem of differential
equations, \emph{Acta. Math., Acad. Sci. Hungar}, 7, pp. 71-94.

\bibitem  {boufoussi2}
{B. Boufoussi,    S. Hajji and   E.  Lakhel.}   {Functional
differential equations in Hilbert spaces driven by a fractional
Brownian motion}. Afrika Matematika, Volume 23, Issue 2, (2012),
173-194.


\bibitem  {boufoussi3}
\textsc{B. Boufoussi and  S.  Hajji.  }  {Neutral stochastic
functional differential equation driven by a  fractional  Brownian
motion in a Hilbert space}. Statist. Probab. Lett. 82,  (2012),
1549-1558.



\bibitem  {carab}
\textsc{T. Caraballo,    M.J. Garrido-Atienza and  T. Taniguchi.}
{The existence and exponential behavior of solutions to stochastic
delay evolution equations with a fractional Brownian motion}.
Nonlinear Analysis 74, (2011), 3671-3684.
\bibitem {danc}
\textsc{T. E. Duncan, J. Jakubowski, and Pasik-Duncan.}
{Stochastic integration for fractional brownian motions in Hilbert
space.} Stoch. Dyn., 6, (2006), 53-75.


\bibitem   {daprato}
\textsc{ G. Da Prato and J.  Zabczyk .}  {Stochastic Equations in
Infinite Dimensions,} \emph{Cambridge University Press, Cambridge
(1992) .}


 \bibitem  {Ferrante1}
 {M. Ferrante and   C. Rovira} . { Stochastic delay differential equations driven
    by fractional Brownian motion with Hurst parameter $H > 1/2$}. Bernoulli 12 (1), (2006), 85-100.

   \bibitem  {Ferrante2}
 {M. Ferrante and  C. Rovira.}   { Convergence of delay differential equations driven by
    fractional Brownian motion with Hurst parameter $H > 1/2$}.  J. Evol. Equ. 10 (4),  (2011), 761-783.




\bibitem {Grec}
\textsc{Grecksch and V. V. Anh. }  A parabolic stochastic
differential equation with fractional brownian motion input.
Statist. Probab. Lett., 41, (1999),  337-346.


\bibitem {kolma}
\textsc{  V. B. Kolmanovskii and   A. D. Myshkis. }   Applied
Theory of Functional Differential Equations, Kluwer Academic,
Dordrecht, (1992).
 \bibitem {kuang}
\textsc{  Y. Kuang . } Delay Differential Equations with
Applications in Population Dynamics, Academic Press, San Diego,
(1993).


\bibitem  {mahmud}
 \textsc{ N.I. Mahmudov. }  Existence and uniqueness results for
neutral SDEs in Hilbert spaces, Stochastic Analysis and
Applications, 24, (2006), 79-95.


%*******************************************************************
\bibitem  {Neuenkirch}
{A. Neuenkirch,    I.  Nourdin,  and  S. Tindel.}   {Delay
equations driven by rough paths}. Electronic Journal of
Probability. Vol.13, (2008), 2031-2068.

\bibitem  {nualart}
\textsc{ D. Nualart.}   { The Malliavin Calculus and Related
Topics, second edition,} Springer-Verlag, Berlin (2006).







\bibitem  {pazy}
 \textsc{A. Pazy.}    {Semigroups of Linear Operators and Applications to Partial Differential Equations}.
   Applied Mathematical Sciences, vol. 44, Springer-Verlag, New York (1983).

\bibitem {peter}
\textsc{ E. E. Peters.} {Fractal Market Analysis.} Wiley New York
(1994).


\end{thebibliography}
\end{document}